\documentclass[11pt]{amsart}

\usepackage{psfrag}

\usepackage{graphicx}

\usepackage{mathtools}

\newcommand{\N}{{\mathbb N}}

\newcommand{\R}{{\mathbb R}}

\newcommand{\Sp}{{\mathbb S}}

\newcommand{\pt}{\partial}

\newcommand{\wh}{\widehat}

\newcommand{\ov}{\overline}

\newcommand{\Teich}{\mathcal{T}}

\newcommand{\QF}{\mathcal{QF}}

\newcommand{\F}{\mathcal{F}}

\newcommand{\Map}{\mathbf{F}}

\newcommand{\MF}{\operatorname{MF}}

\newcommand{\FMF}{\MF_{\dagger}}

\newcommand{\ML}{\operatorname{ML}}

\newcommand{\FML}{\ML^2_{\dagger}}

\newcommand{\MeF}{\mathcal{MF}}

\newcommand{\QD}{\operatorname{QD}}

\newcommand{\q}{\mathbf{q}}

\newcommand{\hor}{\mathbf{H}}

\newcommand{\ver}{\mathbf{V}}

\newcommand{\Ne}{\operatorname{U}}

\newcommand{\Bers}{\mathbf{B}}

\newcommand{\dist}{\mathbf{d}}

\newcommand{\Belt}{\operatorname{Belt}}

\newcommand{\zero}{\mathbf{0}}

\newcommand{\LL}{\operatorname{L}}

\newtheorem{theorem}{\rm\bf Theorem}[section]
\newtheorem{proposition}[theorem]{\rm\bf Proposition}
\newtheorem{lemma}[theorem]{\rm\bf Lemma}

\newtheorem{definition}[theorem]{\rm\bf Definition}
\newtheorem{remark}[theorem]{\rm\bf Remark}

\newtheorem{question}[theorem]{\rm\bf Question}

\begin{document}
\title[Measured foliations at infinity]{Measured foliations at infinity of quasi-Fuchsian manifolds}

\author{ Diptaishik Choudhury \,\,\text{and}\,\, Vladimir Markovi\'c}
\address{\newline YMSC  \newline Tsinghua University   \newline Beijing, China \newline {\&} \newline BIMSA \newline Beijing, China }

\today


\begin{abstract}  Let $(\lambda^+(M),\lambda^-(M))$ denote the pair of measured foliations at the boundary at infinity $\partial_\infty$ of a quasi-Fuchsian manifold $M$. We prove that  $(\lambda^+(M),\lambda^-(M))$ is filling if $M$ is close to being Fuchsian. We also show that given any filling pair $(\alpha_1,\alpha_2)$ of measured foliations, and every small enough $t>0$, the pair $(t\alpha_1,t\alpha_2)$ is realised as the pair of measured foliations at infinity of some quasi-Fuchsian manifold $M$. This answers questions of Schlenker \cite{schlenker} near the Fuchsian locus.
\end{abstract}

\maketitle

\section{Introduction}
\subsection{A word on notation}  Once and for all we fix an orientable closed smooth surface $\Sigma_g$ of genus $g\ge 2$. We let $\Sigma$ and $\ov{\Sigma}$ denote the surface $\Sigma_g$ equipped with the opposite orientations respectively. Throughout the paper we adopt the following (standard) notation:\\

$\F=\text{marked Fuchsian manifolds homeomorphic to  $\Sigma_g \times \R$}$, \\

$\QF=\text{marked quasi-Fuchsian manifolds homeomorphic to  $\Sigma_g \times \R$}$,\\
\vskip .1cm
\noindent
Let $X$ denote a Riemann surface marked by $\Sigma_g$.  We let  $\Sigma_X=\Sigma$ if $X$ has  the same orientation as $\Sigma$, and  $\Sigma_X=\ov{\Sigma}$ if  $X$ has the same orientation as $\ov{\Sigma}$. 
Then:\\

$\QD(X)=\text{the vector space of holomorphic quadratic differentials  on  $X$}$,\\ 

$\Belt (X)=\text{the vector space of Beltrami differentials on $X$}$,\\

$\Teich(\Sigma_X)=\text{the Teichm\"uller space of $\Sigma_X$}$,\\ 

$\QD(\Sigma_X)=\text{the vector bundle $\{\QD(Y)\}_{Y \in \Teich(\Sigma_{X})}$}$,\\

$\QD_0(\Sigma_X)=\{\phi\in \QD(\Sigma_X) \,:\, \phi\not\equiv 0\}$.\\

\vskip .1cm
\noindent
We let $\Teich(\Sigma_g)=\Teich(\Sigma)\sqcup\Teich(\ov{\Sigma})$, and $\QD(\Sigma_g)=\QD(\Sigma)\sqcup\QD(\ov{\Sigma})$.  The Teichm\"uller metric  is denoted by $\dist_\Teich(\cdot,\cdot)$. The vector bundles   
$\QD(\Sigma)$ and $\QD(\ov{\Sigma})$ are isomorphic to the cotangent bundles over 
$\Teich(\Sigma)$, and $\Teich(\ov{\Sigma})$, respectively. 
\vskip .1cm
By $\MeF$ we denote the space of measured foliations on $\Sigma_g$. Two measured foliations fill the surface $\Sigma_g$ if any third  (non-zero) measured foliation has  a non-zero intersection number with at least one of the two foliations. If $A\in \MeF$ we let $[A]$ denote the measure equivalence class of $A$.  We let \\

$\MF=\{[A]:A\in \MeF\},\quad\quad\quad \MF^2=\MF\times \MF$, \\

$\FMF^2=\{([A_1],[A_2]): A_1,A_2\in \MeF, \,\,\text{and}\,\, (A_1, A_2) \,\, \text{fill}\,\, \Sigma_g \}$, \\

$\hor(\phi)$=\text{the equivalence class of the horizontal measured foliation of $\phi\in \QD_0(X)$},\\

$\ver(\phi)$=\text{the equivalence class of the vertical measured foliation of $\phi\in \QD_0(X)$}.\\

\subsection{The mirror surface} The space $\Teich(\Sigma_g)$ is equipped with the natural involution 
$$
\Teich(\Sigma_g)\xrightarrow{X \to \ov{X}} \Teich(\Sigma_g)
$$
which sends $X$ to its mirror image  Riemann surface $\ov{X}$. The mirror map exchanges the  components  
$\Teich(\Sigma)$ and $\Teich(\ov{\Sigma})$. Furthermore, it induces  the linear isomorphism  
$$
\QD(\ov{X}) \xrightarrow{\phi \to \wh{\phi}} \QD(X)
$$
for $X\in \Teich(\Sigma_g)$ as follows. We let $\iota:X\to \ov{X}$ denote the corresponding anti biholomorphic (mirror) map.  Given $\phi \in \QD(\ov{X})$, we let
$$
\wh{\phi}=\overline{(\phi\circ\iota)(\iota')^2},
$$
where $\iota'=\ov{\pt}\iota$. We record the following (obvious)  proposition.
\vskip .1cm
\begin{proposition}\label{prop-vic} Let $\phi\in \QD_0(\Sigma_g)$. Then $\hor(\phi)=\hor(\wh{\phi})$, and 
$\ver(\phi)=\ver(\wh{\phi})$, in $\MF$.
\end{proposition}
\vskip .1cm
\begin{proof} The map $\iota$ between  the marked Riemann surfaces $X$ and $\ov{X}$ induces the identity map on $\Sigma_g$. 
\end{proof}
\vskip .1cm
The map $\iota:X\to \ov{X}$  induces another isomorphism 
$$
\Belt(\ov{X}) \xrightarrow{\mu \to \wh{\mu}} \Belt(X)
$$
by letting
$$
\wh{\mu}=\overline{(\mu\circ\iota)}\frac{\iota'}{\ov{\iota'}}.
$$

\vskip .1cm

\subsection{The Bers unifomization and embedding} The Bers unifomization is the homeomorphism
$$
\Bers:\QF\to \Teich(\Sigma)\times \Teich(\ov{\Sigma})
$$
which sends $(X,\ov{Y})\in \Teich(\Sigma)\times \Teich(\ov{\Sigma})$ to  the marked quasi-Fuchsian manifold $M\in \QF$ such that 
$X\approx \partial_\infty^+M$, and $\ov{Y}\approx\partial_\infty^-M$. Here $\partial_\infty^+M$, and $\partial_\infty^-M$,  denote the two  components of the boundary at infinity of $M$ endowed with the induced complex structures. 
\vskip .1cm

On the other hand, given $X\in \Teich(\Sigma_g)$ we let
$$
\beta_X:\Teich(\ov{\Sigma_X})\to \QD(X)
$$
denote the Bers embedding.

\vskip .1cm

\begin{definition}\label{def-q} We define the maps
$$
\q^+:\QF\setminus \F\to \QD_0(\Sigma) \quad \quad\quad \q^-:\QF\setminus \F\to \QD_0(\ov{\Sigma})
$$
by letting $\q^+(M)=\beta_X(\ov{Y})$, and $\q^-(M)=\beta_{\ov{Y}}(X)$, where $(X,\ov{Y})=\Bers^{-1}(M)$.
\end{definition}
\vskip .1cm

\subsection{The measured foliation at infinity} We have:

\vskip .1cm

\begin{definition}\label{def-main} The measured foliation at infinity of a quasi-Fuchsian manifold $M\in \QF\setminus \F$ is the pair 
$\lambda(M)=(\lambda^+(M),\lambda^-(M))$, where $\lambda^\pm(M)=\hor(\q^\pm(M))$. This defines the map 
$$
\lambda:\QF\setminus \F\to \MF^2.
$$
\end{definition}

\vskip .1cm
The Bers unifomization implies that any pair of marked Riemann surfaces in $\Teich(\Sigma)\times \Teich(\ov{\Sigma})$ can be (uniquely) realised as the boundary at infinity of some  quasi-Fuchsian manifold $M$. It is natural to inquire to which extent this holds if the pair of marked Riemann surfaces  is replaced by the topological data  $\lambda(M)\in  \MF^2$. These types of questions particularly came into focus after Krasnov-Schlenker \cite{k-s} discovered that the variational formula for the renormalised volume at a point $M\in \QF\setminus \F$ only depends on $\lambda(M)$ (also see \cite{schlenker-1}). 

\vskip .1cm
\begin{remark}
The map $\lambda$ is analogous to the map $\ell:\QF\setminus \F\to \ML^2$ where $\ell(M)=(\ell^+(M),\ell^-(M))$, and $\ell^\pm(M)$, are the bending measured laminations  of the  boundary components of the convex core of $M$. Here $\ML$ denotes the space of geodesic measured laminations.  Bonahon-Otal \cite{b-o} completely described the image $\ell(\QF\setminus \F)$. Dular-Schlenker \cite{d-s} showed recently that  $\ell$ is injective.
\end{remark}
\vskip .1cm

In  \cite{schlenker} Schlenker raised the following questions:
\vskip .1cm
\begin{question}\label{question-1} Describe the image $\lambda(\QF\setminus \F)$. 
\end{question}
\vskip .1cm

\begin{question}\label{question-2}  Is  $\lambda(\QF\setminus \F)\subset \FMF^2$? 
\end{question}
\vskip .1cm
\begin{remark} The  inclusion $\ell(\QF\setminus \F)\subset \FML$ is an observation of Thurston.
\end{remark}
\vskip .1cm
\begin{question}\label{question-3}  Does $\lambda(M)$ uniquely determine $M$?
\end{question}
\vskip .1cm

\vskip .1cm

Very little is known regarding these questions. Bonahon  \cite{bonahon} used differentiability of a topological blow-up of  the map $\ell$ at the Fuchsian locus $\F$ to answer the analogous questions (in the context of bending measures) near the Fuchsian locus .  However, it is not known that a  blow-up  of $\lambda$ has such differentiable properties at every point of $\F$ (compare with \cite{choudhury}). In fact, this seems unlikely.

\vskip .1cm

\subsection{The main results}

The main goal of this paper is to provide   answers to  Question \ref{question-1}  and Question \ref{question-2} in a neighbourhood of the Fuchsian locus. This is the content of the following  theorem.

\begin{theorem}\label{thm-main-1} There exists a neighbourhood  $\Ne\subset \QF\setminus \F$ of the Fuchsian locus $\F$ such that 
\begin{enumerate}
\item $\lambda(\Ne)\subset \FMF^2$,
\vskip .1cm
\item for any $(\alpha_1,\alpha_2)\in \FMF^2$ there exists $t_0>0$, depending on  $(\alpha_1,\alpha_2)$, such that  
$(t\alpha_1,t\alpha_2)\in \lambda(\Ne)$ for every $0<t<t_0$. 
\end{enumerate}
\end{theorem}
\vskip .1cm
\noindent	
The proof of the first part of Theorem \ref{thm-main-1} rests on establishing the following property of the maps $\q^+$ and 
$\q^-$.
\vskip .1cm
\begin{theorem}\label{thm-main-2} Suppose $M_n \in \QF\setminus \F$ is a sequence of  quasifuchsian manifolds converging to a  Fuchsian manifold $M\in \F$. Let $\,\,t_n=\dist_{\Teich}(\pt^+_\infty M_n,\ov{\pt^-_\infty M_n})$. There exist a quadratic differential $\phi \in \QD(\pt^+_\infty M)\setminus\zero$,
such that 
$$
\lim\limits_{n\to \infty} \frac{\q^+(M_n)}{t_{n}}= \phi,\quad\quad\quad \lim\limits_{n\to \infty} \frac{\q^-(M_n)}{t_{n}}=-\wh{\phi}.
$$
\end{theorem}
\vskip .1cm

The second  part of Theorem \ref{thm-main-1} is a consequence of the following theorem. By $||\varphi||_1$ we denote the $L^1$-norm of $\varphi\in \QD_0(X)$. We define the  subset $\LL \subset \QD_0(\Sigma)\times [0,\infty)$ by 
$$
\LL=\left\{(\varphi,t): \varphi\in \QD_0(\Sigma),\, 0\le t<\frac{1}{||\varphi||_1}\right\}. 
$$
\vskip .1cm
\begin{theorem}\label{thm-main-3} There exists  a map $\Map: \LL \rightarrow \MF^2$ with the following properties:
\begin{enumerate}
\item $\Map$ is continuous, 
\vskip .1cm
\item  $\Map(\cdot,0): \QD_0(\Sigma)\times\{0\}  \rightarrow \MF^2_\dagger$ is a homeomorphism,
\vskip .1cm 
\item if  $\Map(\varphi,t)=(\alpha_1,\alpha_2)\in \MF^2$ for some $\varphi\in \QD_0(\Sigma)$, and  $0<t<\frac{1}{||\varphi||_1}$,  then there exits $M\in \QF\setminus \F$ such that 
$(t\alpha_1,t\alpha_2)=\lambda(M)$. 
\end{enumerate}
\end{theorem}
\vskip .1cm
\begin{remark} The second property implies  that each pair $\Map(\varphi,0)=(\alpha_1,\alpha_2)$, $\varphi\in \QD_0(\Sigma)$, is filling. We use this to show  that the pair  $\Map(\varphi,t)=(\alpha_1,\alpha_2)$ is also filling providing $t$ is small enough.
\end{remark}

\vskip .1cm

\subsection{A brief outline} Given $X,Y\in \Teich(\Sigma_X)$, we define  quadratic differentials $\Phi(X,Y)\in \QD(X)$, and $\Phi(Y,X)\in \QD(Y)$, so that the harmonic Beltrami differential  $\rho^{-2}_X\ov{\Phi(X,Y)}\in \Belt(X)$, and $\rho^{-2}_Y\ov{\Phi(Y,X)}\in \Belt(Y)$, represent tangent vectors to the Teichm\"uller geodesic arc connecting $X$ with $Y$. Moreover, we choose these tangent vectors so they are pointing to each other. 
This implies that the distance (in $\QD(\Sigma_X)$) between the quadratic differentials $\Phi(X,Y)/\dist_{\Teich}(X,Y)$, and $-\Phi(Y,X)/\dist_{\Teich}(X,Y)$, is  small when $\dist_{\Teich}(X,Y)$ is small.
\vskip .1cm
On the other hand, we prove that the distance (in $\QD(X)$) between $\beta_X(\ov{Y})/\dist_{\Teich}(X,Y)$, and $\Phi(X,Y)/\dist_{\Teich}(X,Y)$, is small when $\dist_{\Teich}(X,Y)$ is small. Putting this together proves Theorem \ref{thm-main-2}. We then use this to prove the first part of Theorem \ref{thm-main-1}.
\vskip .1cm

 The map $\Map$ in Theorem \ref{thm-main-3} is constructed as continuous deformation of the map $\Map(\cdot,0): \QD_0(\Sigma)\times\{0\}  \rightarrow \MF^2_\dagger$. The homeomorphism $\Map(\cdot,0)$ is constructed as the composition of  the Gardiner-Masur homeomorphism  $\gamma:\QD_0(\Sigma)\to \FMF^2$, and the homeomorphism 
$h:\QD_0(\Sigma)\to \QD_0(\Sigma)$ which arises from identifying the tangent space $T_X\Teich(\Sigma)$ with $\QD(X)$ using Teichm\"uller Finsler structure, and the harmonic Beltrami differentials, respectively. The second part of Theorem \ref{thm-main-1} follows  by combining  Theorem \ref{thm-main-3} with some basic lemmas about the degree  of continuous self-maps of spheres.
\vskip .1cm

\section{Harmonic Beltrami differentials and the Bers embedding}
In this section we recall  the notion of a harmonic Beltrami differential and explain its connection with the Bers embedding.
We adopt the following notation.
The vector space $\Belt(X)$ is  equipped with the supremum norm $||\mu||_\infty$, $\mu \in \Belt(X)$. 
We consider two norms on the vector space $\QD(X)$. The first one is  the Bers norm
$$
||\phi||=\max_{p\in X} \rho_X^{-2}(p)|\phi(p)|,\quad\quad  \phi \in \QD(X),
$$
where $\rho_X$ is the density of the hyperbolic metric on $X$.
The second one is the $L^1$-norm
$$
||\phi||_1=\int\limits_X |\phi|.
$$
We also let 
$$
\QD_1(X)=\{\phi\in\QD(X): ||\phi||_1=1 \}.
$$
\vskip .1cm

\subsection{Harmonic Beltrami differentials} We say that $\mu,\nu\in \Belt(X)$ are equivalent if
$$
\int\limits_{X} \mu\phi=\int\limits_{X} \nu\phi
$$
for every $\phi\in \QD(X)$.  The quotient space $\Belt(X)$ is naturally identified with $T_X\Teich(\Sigma_X)$.
The following proposition states that each equivalence class in $\Belt(X)$ contains a unique harmonic Beltrami differential (see \cite{ahlfors}).
\vskip .1cm
\begin{proposition}\label{prop-lab} For every $\mu\in \Belt(X)$ there exists a unique $\Psi(\mu) \in \QD(X)$ such that
\begin{equation}\label{eq-vatos}
\int\limits_{X} \mu\phi=\int\limits_{X} \rho^{-2}_X\ov{\Psi(\mu)}\phi
\end{equation}
for every $\phi \in \QD(X)$. 
\end{proposition}
\vskip .1cm
\subsection{The first derivative of the  Bers embedding}

Suppose $\mu\in \Belt(X)$ with $||\mu||_\infty\le 1$. Let    $f_t:X\to Y_t\in \Teich(\Sigma_X)$, $0\le t<1$, be the path of quasiconformal maps $f_t$ whose Beltrami differential is equal to $t\mu$. Then $t \to Y_t$ is a smooth path in $\Teich(\Sigma_X)$. 
\vskip .1cm
Consider the path $\beta_{\ov{X}}(Y_t)$ in $\QD(\ov{X})$. Bers computed the first derivative of this path at the time $t=0$ (see Section 8 in \cite{bers})
\begin{equation}\label{eq-class}
\frac{d}{dt} \,  \beta_{\ov{X}}(Y_t) \bigg\vert_{t=0}=\wh{\Psi(\mu)}.
\end{equation}
\vskip .1cm
\begin{lemma}\label{lemma-class} For every compact set $K \subset \Teich(\Sigma_g)$ there exist constants
$C=C(K)>0$, and $t_0=t_0(K)$, such that for every $0\le t \le t_0$ the inequality
\begin{equation}\label{eq-mac}
||\beta_{\ov{X}}(Y_t)-t\wh{\Psi(\mu)}||\le Ct^2
\end{equation}
holds assuming $X\in K$.
\end{lemma}
\begin{proof} Since the Bers embedding is a holomorphic  map, and since $Y_t$ is a smooth path in $\Teich(\Sigma_X)$, it follows that  $t\to\beta_{\ov{X}}(Y_t)$ is a smooth  path in $\QD(\ov{X})$. Thus, applying (\ref{eq-class}), and since $K$ is compact, we see that
there exists $t_0=t_0(K)>0$, and $C=C(K)>0$,  such that   (\ref{eq-mac}) holds for $0\le t\le t_0$.
\end{proof}

\vskip .1cm

\section{Comparing $\beta_X(\ov{Y})$ and $\beta_{\ov{Y}}(X)$} 
In this section, we utilise  the notions from the previous section and  define the map  $(X,Y) \to \Phi(X,Y)\in \QD(X)$. Relying on  the comparison between $\Phi(X,Y)$ and $\beta_X(\ov{Y})$, we complete the proof of Theorem \ref{thm-main-2}.

\vskip .1cm

\subsection{The differential $\Phi(X,Y)$} We begin with the following definition.
\vskip .1cm
\begin{definition} For $\varphi\in \QD_1(X)$ we let 
$$
\mu_\varphi=\frac{\ov{\varphi}}{|\varphi|}.
$$
\end{definition} 
\vskip .1cm
Let $X,Y\in\Teich(\Sigma_X)$, and consider the  Teichm\"uller map $f:X\to Y$. The Beltrami differential of $f$ is of the form
\begin{equation}\label{eq-tmap}
\frac{\ov{\pt}f}{\pt f}=k_{XY}\frac{\ov{\varphi}}{|\varphi|}=k_{XY}\mu_\varphi
\end{equation}
for some $\varphi \in \QD_1(X)$, and $0\le k_{XY}<1$.
Here 
\begin{equation}\label{eq-k}
\frac{1}{2}\log \frac{1+k_{XY}}{1-k_{XY}}=\dist_\Teich(X,Y).
\end{equation}

\vskip .1cm
\begin{definition} Define $\Phi(X,Y)\in \QD(X)$ by letting 
$$
\Phi(X,Y)=\Psi\big(k_{XY}\mu_\varphi)=k_{XY}\Psi\big(\mu_\varphi), 
$$
where $\Psi(k_{XY}\mu_\varphi)$ is the quadratic differential defined by Proposition \ref{prop-lab}.
\end{definition}
\vskip .1cm

\subsection{Comparing $\Phi(X,Y)$ and $\Phi(Y,X)$} In the following lemma we compare the limits of suitably normalised 
differentials $\Phi(X,Y)$, and $\Phi(Y,X)$, respectively.

\begin{lemma}\label{lemma-lim} Suppose $X, X_n,Y_n\in \Teich(\Sigma)$, are such that $X_n\ne Y_n$ for every $n\in \N$,  and that both  sequences $X_n$, and $Y_n$, converge to $X$. Then there exits $\varphi\in \QD_1(X)$   so that (after passing to a subsequence) we have
$$
\lim_{n\to \infty} \frac{\Phi(X_n,Y_n)}{k_{X_{n}Y_{n}}}=\Psi(\mu_\varphi) \quad \quad \quad \lim_{n\to \infty} \frac{\Phi(Y_n,X_n)}{k_{X_{n}Y_{n}}}=-\Psi(\mu_\varphi).
$$
\end{lemma}

\begin{proof} Consider the Teichm\"uller maps $f_n:X_n\to Y_n$, and $g_n:Y_n\to X_n$, with the Beltrami differentials
$$
\frac{\ov{\pt}f_n}{\pt f_n}=k_n\frac{\ov{a_n}}{|a_n|}=k_n\mu_{a_{n}},\quad\quad \frac{\ov{\pt}g_n}{\pt g_n}=k_n\frac{\ov{b_n}}{|b_n|}=k_n\mu_{b_{n}},
$$
where  $a_n\in \QD_1(X_n)$, and $b_n\in \QD_1(Y_n)$. Here we use the notation  
$k_n=k_{X_{n}Y_{n}}=k_{Y_{n}X_{n}}$.
\vskip .1cm
After passing to a subsequence, we may assume that $a_n\to a$, and $b_n\to b$, where $a,b\in \QD_1(X)$.
Thus,  $\mu_{a_{n}}\to \mu_a$, and  $\mu_{b_{n}}\to \mu_b$, in the bundle $\{\Belt(Z)\}_{Z\in \Teich(\Sigma_{X})}$, when $n\to \infty$.
\vskip .1cm

On the other hand, the Beltrami differentials $\mu_{a_{n}}$ and  $\mu_{b_{n}}$ represent the unit  vectors   $u_n\in T_{X_{n}}\Teich(\Sigma_{X})$, and $v_n\in T_{Y_{n}}\Teich(\Sigma_{X})$, respectively. These  vectors $u_n$ and $v_n$ are tangent  to the Teichm\"uller geodesic arc connecting $X_n$ and $Y_n$, and are pointing towards each other. Therefore, there exists a unit vector $w \in T_{X}\Teich(\Sigma_{X})$ such that $u_n\to w$, and $v_n\to -w$, where the convergence is in the bundle 
$T\Teich(\Sigma_{X})$.
\vskip .1cm

But, the vector $w$ is represented by $\mu_a$, and the vector $-w$ is represented by $\mu_b$. It follows that 
$a=-b$. Set $\varphi=a$. We have shown that 
$$
\lim_{n\to \infty} \Psi(\mu_{a_{n}})=\Psi(\mu_{\varphi}),\quad \quad\quad  \lim_{n\to \infty}\Psi(\mu_{b_{n}})= 
\Psi(\mu_{-\varphi})=-\Psi(\mu_{\varphi}).
$$
This proves the lemma.
\end{proof}

\vskip .1cm

\subsection{Comparing $\beta_X(\ov{Y})$ and  $\Phi(X,Y)$} In this subsection we compare $\beta_X(\ov{Y})$ with $\beta_{\ov{Y}}(X)$ when $\dist_\Teich(X,Y)$ is  small.
\vskip .1cm
\begin{lemma}\label{lemma-lim-1} Suppose $X, X_n,Y_n\in \Teich(\Sigma)$, are such that $X_n\ne Y_n$ for  $n\in \N$,  and that both  $X_n$, and $Y_n$, converge to $X$. There exits $\varphi\in \QD_1(X)$,  so that (after passing to a subsequence) we have
$$
\lim_{n\to \infty} \frac{\beta_{X_{n}}(\ov{Y}_n)}{ \dist_{\Teich(X_{n},Y_{n})}}=\Psi(\mu_\varphi) \quad \quad \quad \lim_{n\to \infty} \frac{\beta_{\ov{Y}_{n}}(X_n)}{ \dist_{\Teich(X_{n},Y_{n})}}=-\wh{\Psi(\mu_\varphi)}.
$$
\end{lemma}

\begin{proof} From (\ref{eq-mac}) we know that for some constant $C_1=C_1(K)$, the inequalities
\begin{equation}\label{eq-q}
\big\vert\big\vert \frac{\beta_{X_{n}}(\ov{Y}_n)}{k_{n}} -\frac{\Phi(X_n,Y_n)}{k_{n}}\big\vert\big\vert\le C_1k_n,
\end{equation}
and
\begin{equation}\label{eq-q}
\big\vert\big\vert \frac{\beta_{\ov{Y}_{n}} (X_n)}{k_{n}} -\frac{\wh{\Phi(Y_n,X_n)}}{k_{n}}\big\vert\big\vert\le C_1k_n,
\end{equation}
hold. 
Combining this with Lemma \ref{lemma-lim}  implies that 
$$
\lim_{n\to \infty} \frac{\beta_{X_{n}}(\ov{Y}_n)}{k_{n}}=\Psi(\mu_\varphi) \quad \quad \quad \lim_{n\to \infty} \frac{\beta_{\ov{Y}_{n}}(X_n)}{ k_{n}}=-\wh{\Psi(\mu_\varphi)},
$$
for some $\varphi\in \QD_1(X)$. Together with 
\begin{equation}\label{eq-k-111}
\lim_{n\to \infty} \frac{\dist_\Teich(X_{n},Y_{n})}{k_{n}}=1,
\end{equation}
this proves the lemma. Note that (\ref{eq-k-111}) follows from (\ref{eq-k}).
\end{proof}

\vskip .1cm

\subsection{Proof of Theorem \ref{thm-main-2}}

Suppose that $M_n \in \QF\setminus \F$ is a sequence of  quasifuchsian 3-manifolds converging to a  Fuchsian manifold $M\in \F$. We let 
$$
X_n=\pt^+_\infty M_n \quad\quad\quad \ov{Y}_n=\pt^-_\infty M_n.
$$
Then $\q^+(M_n)=\beta_{X_{n}}(\ov{Y}_n)$, and $\q^-(M_n)=\beta_{\ov{Y}_{n}}(X_n)$.
The proof of Theorem  \ref{thm-main-2}  follows from Lemma \ref{lemma-lim-1}.

\vskip .1cm

\section{Constructing $\Map$ and the proof of Theorem \ref{thm-main-3}}

In this section we construct the map $\Map: \LL \rightarrow \MF^2$, and prove Theorem \ref{thm-main-3}.
Let $\gamma:\QD(\Sigma_g)\to \FMF^2$ be the map given by $\gamma(\phi)=(\hor(\phi),\ver(\phi))$. As it is well known,   combining the results from Kerckhoff \cite{ker},  Gardiner-Masur \cite{g-m}, and Wentworth \cite{wentworth},  shows that $\gamma$ is a homeomorphism.

\vskip .1cm

\subsection{Constructing $\Map$}

For $\varphi\in \QD_0(X)$, we let $\varphi^1=\varphi/||\varphi||_1$. Then $\varphi^1\in \QD_1(X)$, and we consider the corresponding Beltrami differential $\mu_{\varphi^{1}}\in \Belt(X)$. Define 
$$
h:\QD_0(\Sigma)\to \QD_0(\Sigma)
$$
by $h(\varphi)=||\varphi||_1\Psi(\mu_{\varphi^{1}})$. Clearly, $h$ is a (homogeneous) homeomorphism.

\vskip .1cm

Let    $f_t:X\to Y_t\in \Teich(\Sigma)$, $0\le t<\frac{1}{||\varphi||_{1}}$, be the path of quasiconformal maps $f_t$ whose Beltrami differential is equal to $(t||\varphi||_{1})\mu_{\varphi^{1}}$.  We define $\Map$ by letting 
$$
\Map(\varphi,t)=\frac{1}{t}\left(  \hor(\beta_{X}(\ov{Y}_t )),\hor(\beta_{\ov{Y}_{t}}(X)) \right).
$$
This defines the map $\Map$ on $\LL_0=\LL\setminus \big(\QD_0(\Sigma)\times \{0\}$\big). It remains to show that $\Map$ extends continuously to the entire domain $\LL$.

\vskip .1cm

\subsection{Proof of Theorem \ref{thm-main-3}} 

We see from the definition of  $\lambda$ that  if $(\alpha_1,\alpha_2)\in \Map\left(\QD_0(\Sigma)\times \{t\}\right)$, then $(t\alpha_1,t\alpha_2)\in \lambda(\QF\setminus \F)$. This proves the property (3) of $\Map$.

\vskip .1cm
The map $\Map$ is continuous on $\LL_0$.  To finish the proof of the 
theorem we need to prove that $\Map$ is continuous on $\LL$, and that $\Map(\phi,0)$ is a homeomorphism.  Both statements follow from the following lemma.
\vskip .1cm

\begin{lemma} Let  $\varphi_n, \varphi\in \QD_0(\Sigma)$, and $t_n>0$, $n\in \N$. Suppose that  $\varphi_n\to \varphi$ in $\QD_0(\Sigma)$, and $t_n\to 0$, when $n\to \infty$. Then $\Map(\varphi_n,t_n)\to (\gamma\circ h)(\varphi)$, when $n\to \infty$.
\end{lemma}
\begin{proof} Suppose $X_n,X\in \Teich(\Sigma)$ are such that $\varphi_n\in \QD_0(X_n)$, and $\varphi\in \QD_0(X)$. Then $X_n\to X$ in $\Teich(\Sigma)$. Let $Y_n\in \Teich(\Sigma)$ be such that 
$$
\Map(\varphi_n,t_n)=\frac{1}{t_{n}}\left(  \hor(\beta_{X_{n}}(\ov{Y}_n )),\hor(\beta_{\ov{Y}_{n}}(X_n)) \right) .
$$
Note that 
$$
\lim_{n\to \infty} \frac{\dist_{\Teich}(X_n,Y_n)}{||\varphi_{n}||_{1}t_{n}}=1. 
$$
\vskip .1cm
\noindent
It now follows from Lemma \ref{lemma-lim-1}  that 
$$
\lim_{n\to \infty}\frac{\beta_{X_{n}}(\ov{Y}_n )}{||\varphi_{n}||_{1}t_{n}}=\Psi(\mu_{\varphi^1})=\frac{h(\varphi)}{||\varphi||_1 } ,
$$
and
$$
\lim_{n\to\infty} \frac{\beta_{\ov{Y}_{n}}(X_n)}{t_{n}}= -\wh{\Psi(\mu_{\varphi^1})}=-\frac{\wh{h(\varphi)}}{||\varphi||_1}.
$$
Combining this with Proposition \ref{prop-vic} proves the lemma.

\end{proof}

\vskip .1cm

\section{Proof of Theorem \ref{thm-main-1}}

We combine the fact that $\gamma:\QD(\Sigma_g)\to \FMF^2$ is a homeomorphism   with Theorem \ref{thm-main-2} to prove the first part of Theorem \ref{thm-main-1}.

\subsection{Proof of Theorem \ref{thm-main-1}: Part I}   
We need to prove that there exists a neighbourhood  $\Ne \subset \QF\setminus \F$ of the Fuchsian locus $\F$ such that 
$\lambda(\Ne)\subset \FMF^2$. The proof is by contradiction. 

\vskip .1cm

If there is no neighbourhood  $\Ne \subset \QF\setminus \F$ of the Fuchsian locus $\F$ such that 
$\lambda(\Ne)\subset \FMF^2$, then there exist a sequence $M_n\in \QF\setminus \F$, and $M\in \F$, such that 
$M_n\to M$, and such that $\lambda(M_n)\notin  \MF^2_\dagger$. From now onwards, we assume that such a sequence exists.

\vskip .1cm

Let $t_n=\dist_{\Teich}(\pt^+_\infty M_n,\pt^-_\infty M_n)$. Then by Theorem \ref {thm-main-2} there exists a  quadratic differential $\phi \in \QD_0(\pt^+_\infty M)$,
such that 
\begin{equation}\label{eq-cons}
\lim\limits_{n\to \infty} \frac{\q^+(M_n)}{t_{n}}= \phi,\quad\quad\quad \lim\limits_{n\to \infty} \frac{\q^-(M_n)}{t_{n}}= -\wh{\phi}.
\end{equation}
This implies that 
$$
\lim_{n\to\infty}\hor\left(\frac{\q^+(M_n)}{t_{n}} \right)= \hor(\phi)\quad\quad\quad    \lim_{n\to\infty}  \hor\left( \frac{\q^-(M_n)}{t_{n}}\right)=\ver(\phi).
$$
It is well known (see Lemma 5.3 in \cite{g-m}) that $(\hor(\phi),\ver(\phi))\in \MF^2_\dagger$. Since $\MF^2_\dagger$ is an open subset of $\MF^2$, we conclude that 
$\lambda(M_n)=\big(\hor(\q^+(M_n)),\hor(\q^-(M_n)\big)\in \MF^2_\dagger$ for $n$ large enough. This contradicts our assumption and the proof is complete.
\vskip .1cm

\subsection{Continuous deformations of  identity maps} To prove the second part of Theorem \ref{thm-main-1}, we need the following auxiliary result. Let $B(r)\subset \R^n$ denote the closed  ball in $\R^n$ of radius $r>0$ which is centred at the origin in $\R^n$.  
\vskip .1cm

\begin{lemma}\label{lemma-aux-0} Let $t_0>0$, and  suppose $f:B(1)\times [0,t_0] \to \R^n$ is a continuous map such that $f(\cdot,0)$ is the identity map. Let $x_0\in B\big(\frac{1}{2}\big)$. Then  there exists $0<t_1\le t_0$
such that $x_0 \in f(B(1)\times\{t\})$ for every $0\le t\le t_1$.
\end{lemma}

\begin{proof} Our initial goal is to extend the map $f:B(1)\times [0,t_0]\to \R^n$ to $\overline{\R^n}\times [0,t_0]\to \overline{\R^n}$. 
We first extend  $f$ to $B(2)$ as follows. Let $r(x,t)=f(x,t)-x$. Note that $r(x,t)\to 0$ uniformly in $t$, and $x\in B(1)$. Let 
$$
f(x,t)=\begin{cases} 
      f(x,t), & |x| \le 1 \\
      x +(2-|x|)r\big(\frac{x}{|x|},t \big)  & 1\le |x|\le 2 .
\end{cases}
$$
Note that the new $f$ is well defined and continuous on $B(2)\times [0,t_0]$, and  $f(\cdot,0)$ is the identity map on $B(2)$. Moreover, $f(x,t)=x$ for every $t$ assuming $|x|=2$. 
\vskip .1cm

Next, we extend the definition of $f$ to the sphere $\Sp^{n}=\ov{\R^{n}}$ by  inversion. Set 
$$
f(x,t)=\begin{cases} 
      f(x,t), & |x| \le 2 \\
(f(x^*,t))^* & 2\le |x| .
\end{cases}
$$
Here $x\to x^*$ is the inversion map of the sphere $\Sp^{n}$ which maps $B(2)$ onto its complement, and which is equal to the identity on the boundary of $B(2)$. We have now constructed a continuous  map 
$f:\Sp^{n}\times [0,t_0]\to \Sp^{n}$ such that $f(x,0)=x$ for every $x \in \Sp^{n}$. By continuity, for every $t\in [0,t_0]$ the map $f(\cdot,t):\Sp^{n}\to \Sp^{n}$ is of degree one. In particular, each $f(\cdot,t):\Sp^{n}\to \Sp^{n}$  is surjective. 

\vskip .1cm

Consider the point $x_0\in B\big(\frac{1}{2}\big)\subset \overline{\R^n}=\Sp^n$. Since $f(\cdot,t):\Sp^{n}\to \Sp^{n}$  is surjective we conclude that for every $t\in [0,t_0]$ there exists $x_t\in \Sp^n$ such that $f(x_t,t)=x_0$.  On the other hand, $f(x,0)=x$ for every $x\in \Sp^{n}$. Thus, there exists $0<t_1\le t_0$ so that if $f(x,t)=x_0$ for some $0\le t\le t_1$, then $|x|<1$. Therefore, we conclude that $x_t\in B(1)$ when $0\le t\le t_1$.
This completes the proof.
\end{proof}

\subsection{Proof of Theorem \ref{thm-main-1}: Part II}

Set $\Map_0(\cdot)=\Map(\cdot,0)$. Thus, $\Map_0:\QD_0(\Sigma)\to \FMF^2$ is a homeomorphism. 
Let $\alpha=(\alpha_1,\alpha_2)\in \FMF^2$, and let $\psi=\Map^{-1}_0(\alpha)$. Choose embedded closed balls $B_\psi \subset \QD_0(\Sigma)$, and $B_\alpha\subset \FMF^2$, containing $\psi$ and $\alpha$ respectively in their interiors, and such that  $\Map_0(B_\psi)=B_\alpha$. 

\vskip .1cm

Since $B_\psi$ is a compact subset of $\QD_0(\Sigma)$, there exists $t_0$ such that $B_\psi \times [0,t_0]\subset \LL$.
Let $B'_\alpha$ be a strictly larger  open ball which is embedded in $\FMF^2$, and which contains the closed ball $B_\alpha$. Then there exists $0<t_1\le t_0$ so that $\Map_0\big(B_\psi\times\{t\}\big)\subset B'_\alpha$, for every $0\le t\le t_1$.
\vskip .1cm
After finding suitable embeddings  $e_1:B_\psi\to \R^n$, and $e_2:B'_\alpha\to \R^n$, we can assume that $e_1(B_\psi)=e_2(B_\alpha)=B$, and  $e_1(\psi)=e_2(\alpha)=0$. Moreover, we may assume that  $e_2\circ \Map_0 \circ e^{-1}_1:B\to B$ is the identity map. 
\vskip .1cm

Now, for $0\le t\le t_1$, the map $f:B\times [0,t_1]\to \R^n$, given by $f=e_2\circ \Map \circ e^{-1}_1$, is well defined and it satisfies the assumptions of Lemma \ref{lemma-aux-0}. Thus, from Lemma \ref{lemma-aux-0} we conclude that there exists $0<t_2\le t_1$ such that $\alpha\in \Map\big(\QD_0(\Sigma)\times \{t\}\big)$ for every $0<t<t_2$. Combining this with the property (3) from  Theorem \ref{thm-main-3} yields the proof of the theorem.


\begin{thebibliography}{99}  

\bibitem{ahlfors} L. Ahlfors, \textsl{Lectures on Quasiconformal Mappings.} 
American Mathematical Society, Providence, RI, (2006)

\bibitem{bers} L. Bers, \textsl{A non-standard integral equation with applications to quasiconformal mappings.} 
Acta Math. 116 (1966), 113-134.

\bibitem{b-o} F. Bonahon, J-P. Otal, \textsl{Laminations mesur\'ees de plissage des vari\'et\'es hyperboliques de dimension
3.}  Ann. Math., 160, 1013-1055, (2004)

\bibitem{bonahon} F. Bonahon,  \textsl{Kleinian groups which are almost {F}uchsian.} Journal f{\"u}r die Reine und Angewandte Mathematik,  vol  587,  1-15, (2005)
		
\bibitem{choudhury} 	D. Choudhury, \textsl{Measured foliations at infinity of quasi-{F}uchsian manifolds near the {F}uchsian locus.}
{https://arxiv: 2111.01614} (2021)		

\bibitem{d-s} B. Dular, J-M. Schlenker, \textsl{Convex co-compact hyperbolic manifolds are determined by their pleating lamination.}
arXiv:2403.10090, (2024)

\bibitem{g-m} F.  Gardiner, H. Masur, \textsl{Extremal length geometry of Teichm\"uller space.} 
Complex Variables, Theory Appl., vol 16, 209-237, (1991)

\bibitem{ker} S. Kerckhoff, \textsl{Lines of minima in teichm\"uller space.} Duke Mathematical Journal,
65(2) (1992)
		
\bibitem{k-s} K. Krasnov, J-M. Schlenker,  \textsl{A symplectic map between hyperbolic and complex {Teichm{\"u}ller} theory.}
Duke Mathematical Journal, vol 150, No 2, 331-356, (2009)
		
\bibitem{schlenker}	J-M.  Schlenker, \textsl{Notes on the Schwarzian tensor and measured foliations at infinity of quasi-{F}uchsian manifolds.} arXiv:1708.01852, (2017)

\bibitem{schlenker-1}J-M.  Schlenker, \textsl{Volumes of quasifuchsian manifolds.}  Surveys in Differential Geometry, 25:1(2020), 319-353

\bibitem{wentworth} R.  Wentworth, \textsl{Energy of Harmonic Maps and Gardiner's Formula.} Contemporary Mathematics,
vol 432, (2007) 

\end{thebibliography}
\end{document}